\documentclass[english]{article}
\usepackage{preamble}

\begin{document}

\title{The Monadic Tower for $\infty$-Categories}
\maketitle

\begin{abstract}
    Every right adjoint functor between presentable $\infty$-categories is shown to decompose canonically as a \textit{coreflection}, followed by, possibly transfinitely many,  \textit{monadic functors}. Furthermore, the coreflection part is given a presentation in terms of a functorial iterated colimit.  Background material, examples, and the relation to homology localization and completion are discussed as well. 
\end{abstract}

\begin{figure}[H]
    \centering{}
    \includegraphics[scale=0.129]{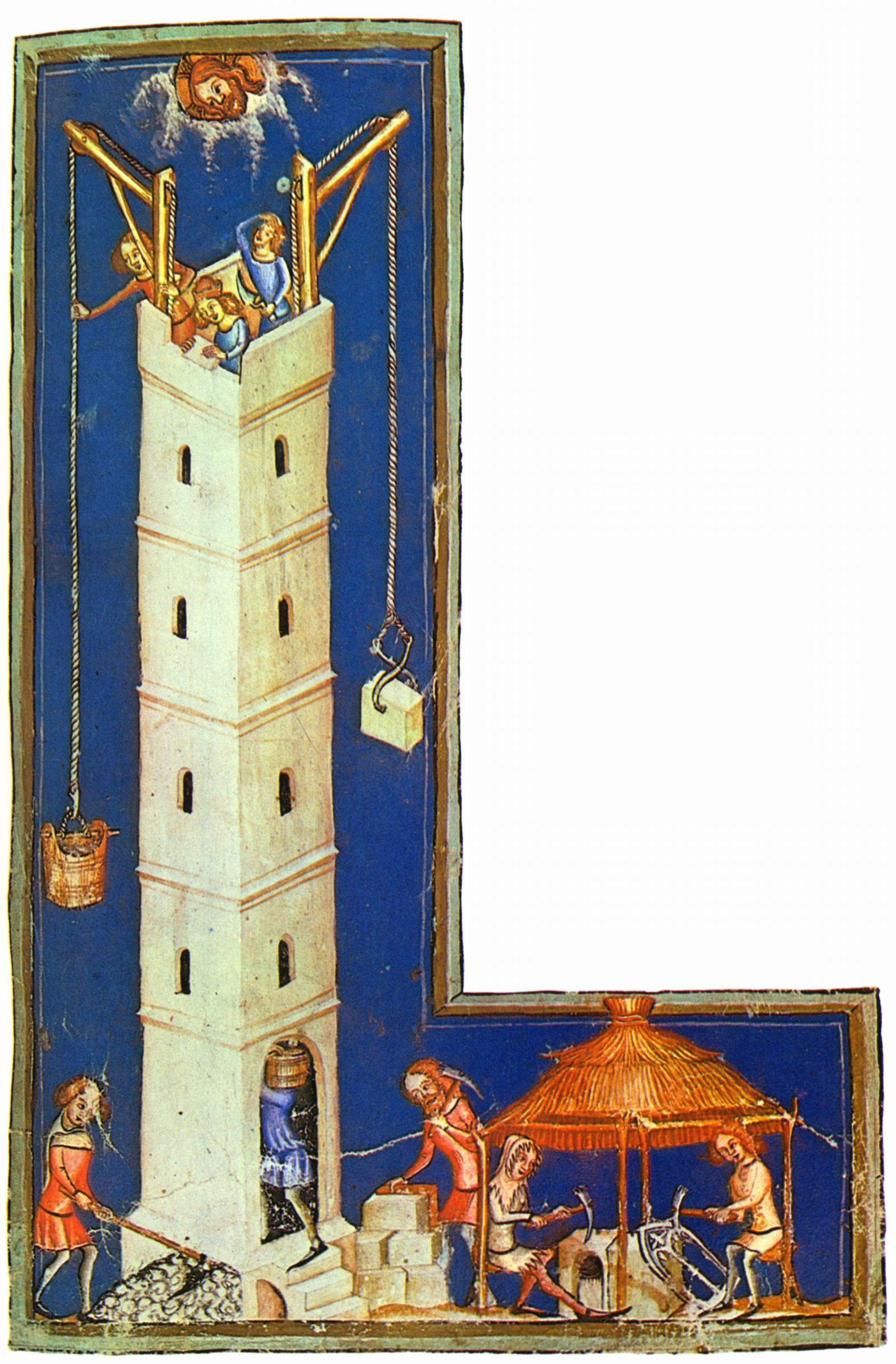}
    \caption*{Weltchronik in Versen, Szene: \textit{Der Turmbau zu Babel} (c. 1370s). Depiction of the construction of the tower of Babel.}
\end{figure}

\tableofcontents{}

\section{Introduction}
\subsubsection{Background}

The theory of \textit{monads} is a general categorical framework that axiomatizes the notion of an \textit{algebraic structure}. 
Loosely speaking, a monad $T$ on an $\infty$-category $\cC$ encodes a collection of formal operations with ``arities in $\cC$'' and certain relations among their compositions. 
For an object $X\in\cC$, a structure of a $T$-algebra on $X$ is a realization of this collection of \textit{formal} operations as \textit{actual} operations on $X$, such that the required relations hold (see \cite{berger2012monads}, for an elaboration of this perspective). This framework encompasses many of the familiar algebraic structures in ordinary and $\infty$-category theory, including, for example, all those which arise from $\infty$-operads.\footnote{It also includes some less ``algebraically looking'' examples, such as the structure of a compact Hausdorff topology on a set, which arises via the \textit{ultrafilter monad} (see, for example, \cite{leinster2013codensity}).} 

The functor from the $\infty$-category of $T$-algebras back to $\cC$, which forgets the $T$-algebra structure, is always a conservative right adjoint. A functor to $\cC$ that is equivalent to such a forgetful functor, for some monad $T$ on $\cC$, is called \textit{monadic}. The famous \textit{monadicity theorem} (due to Barr-Beck for ordinary categories and Lurie for $\infty$-categories) characterizes the monadic functors among conservative right adjoints as those which in addition preserve a certain special type of colimits. This gives an intrinsic criterion for recognizing when an abstract functor is one which ``forgets algebraic structure''. It is interesting to note that monadic functors are not closed under composition. The following is a classical example:

\begin{exa}\label{Ex_Cat}
    Let $\Cat$ be the ordinary category of (small) categories and functors and let $\Grph$ be the ordinary category of (small) reflexive directed multi-graphs. The functor $G''\colon \Cat \to \Grph$, that forgets the composition and remembers only the  objects, morphisms, and identities, is monadic. Similarly, the functor $G'\colon \Grph \to \Set$ that takes a reflexive directed multi-graph to the set of its edges, is monadic as well. However, the composition
    \[
        \begin{tikzcd}
        	\Cat & \Grph & \Set,
        	\arrow["{G''}", from=1-1, to=1-2]
        	\arrow["{G'}", from=1-2, to=1-3]
        	\arrow["G"', curve={height=15pt}, from=1-1, to=1-3]
        \end{tikzcd}
    \]
    is \textit{not} monadic.
\end{exa}

The ``two-step monadicity'' of the functor $G\colon \Cat \to \Set$ in the above example can be understood as follows. A category can be thought of as consisting of a set $A$ of ``morphisms'' endowed with an additional structure. First, one has the unary operators $s,t\colon A \to A$, which map a given morphism to the identity morphisms of its source and target respectively, and which satisfy
\[
    s^2=s\;,\;ts=s\;,\;t^2=t\;,\;st=t.
\]
This structure makes $A$ precisely into a reflexive directed multi-graph (on the set of identity morphisms). Having that, one has the binary composition operation
\[
    \circ \colon 
    A\; {}_s\!\!\times_t A \longrightarrow A,
\]
which is required to satisfy the associativity and unitality identities.
The failure of the forgetful functor 
$G\colon \Cat \to \Set$
to be monadic is closely related to the fact that this composition operation is only \textit{partially defined}, as the source and target operators $s$ and $t$ are  needed to specify its domain of definition. In general, a composition of several monadic functors can still be thought of as encoding an \textit{essentially algebraic structure} on the target (in the sense of \cite{freyd1972aspects}). Namely, one that is given by a collection of recursively \textit{partially defined} operations, whose domain and range are specified by equations involving previously defined operations.

\subsubsection{Main results}

Working backwards, under mild cocompleteness assumptions, every right adjoint functor can be universally approximated by a monadic one. By iterating this process transfinitely, we get a canonical factorization of every right adjoint functor through a transfinite composition of monadic functors. More precisely, 
given a right adjoint functor $G\colon\cD \to \cC$  of (large) $\infty$-categories, such that $\cD$ admits small sifted colimits, we obtain a diagram of $\infty$-categories under $\cD$ indexed by the linearly ordered (large) set of small ordinals
\[
    \cD \lto 
    (\dots \lto \cC_\alpha \lto \dots \lto \cC_2 \lto \cC_1 \lto \cC_0 = \cC).
\]
This diagram satisfies, in particular, that for every ordinal $\alpha$, the functor 
$\cC_{\alpha+1} \to \cC_\alpha$ is monadic and for every limit ordinal $\alpha$, we have 
$\cC_\alpha \iso  \invlim_{\beta<\alpha}\cC_\beta.$
We refer to this construction as the \textit{monadic tower} of $G$ (see \Cref{Def_Monadic_Tower} for details). 

To study the \textit{convergence} of the monadic tower, we further impose the assumption of \textit{presentability}. The main goal of this paper is to prove the following:
\begin{thm*}[Monadic Convergence, {\ref{Monadic_Convergence}}]
    Let $G\colon\cD \to \cC$ be a functor between presentable $\infty$-categories, which admits a left adjoint $F$.
    \begin{enumerate}
        \item The monadic tower of $G$ stabilizes for $\alpha\gg0$ on $\cC_\infty := \invlim_\alpha \cC_\alpha$.
        \item The induced functor 
            \(
                \cD \to \cC_\infty
            \)
            admits a fully faithful left adjoint $\cC_\infty \into \cD$.
        \item  The image of $\cC_\infty \into \cD$ is the subcategory $\cD_\infty\sseq \cD$ generated under colimits by $F(\cC)\sseq \cD$. 
    \end{enumerate}
\end{thm*}

In particular, if the functor $G$ is \textit{conservative}, then the coreflection $G_\infty\colon\cD \to \cD_\infty$ is an \textit{equivalence}. That is, the monadic tower of $G$ converges to $G$ itself. Thus, we get the following:
\begin{corollary}[Transfinite Monadicity, \ref{Trans_Monad}]
    A right adjoint functor between presentable $\infty$-categories is conservative, if and only if it is a transfinite composition of monadic functors.
\end{corollary}
More informally, among functors of presentable $\infty$-categories, the conservative right adjoints are precisely those which forget \textit{essentially algebraic structure}. 

For a general, non-conservative, functor $G\colon \cD \to \cC$, the monadic convergence theorem provides a factorization of $G$ as 
\[
    \cD \oto{G_\infty} \cD_\infty \oto{\:\cl{G}\:} \cC,
\]
where $G_\infty$ admits a fully faithful left adjoint $F_\infty$ and $\cl{G}$ is a transfinite composition of monadic functors. As monadic functors are conservative and the latter are closed under composition, it follows that $\cl{G}$ is conservative as well. We show that this realizes a general ``coreflection-conservative'' factorization system on $\Prr$ and that in every such factorization, $F_\infty(\cD_\infty) \sseq \cD$ is the full subcategory generated under colimits by $F(\cC)$, where $F$ is the left adjoint of $G$ (\Cref{PrR_Factor_Sys}). 
Thus, for an object $Y\in\cD$, the counit map $F_\infty G_\infty (Y) \to Y$ provides a universal approximation of $Y$ from the left, by an object which is an iterated colimit of objects in $F(\cC)$. The realization of $G_\infty$ as the limit of the monadic tower of $G$ provides a canonical and functorial such presentation. Namely, we have the following:

\begin{corollary}[Colocalization Sequence, \ref{Iterated_Colimit_Functorial}]
    Given $G\colon \cD \to \cC$ in $\Prr$, the counits of the adjunctions 
    \[
        \begin{tikzcd}
        	{G_\alpha:\cD} & {\cC_\alpha: F_\alpha,}
        	\arrow[shift right=1, shorten <=1pt, shorten >=1pt, from=1-1, to=1-2]
        	\arrow[shift right=1, shorten <=1pt, shorten >=1pt, from=1-2, to=1-1]
        \end{tikzcd}
    \]
    in the monadic tower of $G$, assemble into a transfinite sequence in $\cD_{/Y}$,
    \[
   \begin{tikzcd}
    	{F_0G_0(Y)} & {F_1G_1(Y)} & F_2G_2(Y) & \dots & {F_\alpha G_\alpha(Y)} & \dots
    	\\ &&&& Y
    	\arrow[from=1-1, to=1-2]
    	\arrow[from=1-2, to=1-3]
    	\arrow[from=1-3, to=1-4]
    	\arrow[from=1-4, to=1-5]
    	\arrow[from=1-5, to=1-6]
    	\arrow[from=1-5, to=2-5]
     	\arrow[curve={height=12pt}, from=1-1, to=2-5]
     	\arrow[curve={height=9pt}, from=1-2, to=2-5]
     	\arrow[curve={height=9pt}, from=1-3, to=2-5]
    \end{tikzcd}
    \]
    which stabilizes on $F_\infty G_\infty (Y) \to Y$. Moreover, for each ordinal $\alpha$, the term $F_\alpha G_\alpha (Y)$ is canonically a colimit of objects in $F_\beta(\cC)$ for $\beta <\alpha$.    
\end{corollary}

\subsubsection{Localization and completion}

The inspiration for this work comes from the subject of homology localizations and completions of spaces in homotopy theory. More concretely, in \cite{dror1977long}, Dwyer and Farjoun\footnote{The mathematician formerly known as Emmanuel \textit{Dror}, has changed his name (back) to Emmanuel \textit{Farjoun} (the words `dror' and `farjoun' mean `freedom' in Hebrew and Arabic respectively).} construct for each space $X$ and a ring $R$, which is either $\FF_p$ or a subring of $\QQ$, a functorial transfinite tower of spaces under $X$
\[
   \begin{tikzcd}
    	& X \\
    	\dots & {T_\alpha X} & \dots & {T_2X} & {T_1X} & {T_0X,}
    	\arrow[from=2-5, to=2-6]
    	\arrow[from=2-2, to=2-3]
    	\arrow[from=2-1, to=2-2]
    	\arrow[from=1-2, to=2-2]
    	\arrow[curve={height=-9pt}, from=1-2, to=2-5]
    	\arrow[curve={height=-12pt}, from=1-2, to=2-6]
    	\arrow[from=2-4, to=2-5]
    	\arrow[from=2-3, to=2-4]
    	\arrow[curve={height=-6pt}, from=1-2, to=2-4]
    \end{tikzcd}
\]
which they call the \textit{long homology localization tower} of $X$ with respect to $R$. Moreover, they show that this tower always stabilizes on the \textit{$R$-localization} of $X$ in the sense of Bousfield (see \Cref{Ex_Homology_Loc}). They also observe that the first few terms of this tower are familiar constructions in terms of the \textit{$R$-completion} of $X$ in the sense of Bousfield-Kan (see \Cref{Unstable_Nil_Completion}). These ideas were further studied in \cite{casacuberta1999localizations}.

The long homology localization sequence can be deduced from our setting, or rather its \textit{dual}, as follows. For every \textit{ring spectrum} $R$, the composition
\[
    \Spc \oto{\;\;\Sigma^\infty_+\;\;} 
    \Sp \oto{R\otimes(-)}
    \Mod_R(\Sp)
\]
is a \textit{left} adjoint functor between presentable $\infty$-categories. Thus, we can construct the associated \textit{comonadic tower}, by formally dualizing the construction of the monadic tower of a right adjoint functor. The localization sequence of Dwyer and Farjoun can then be obtained from this comonadic tower by the same procedure that produces the colocalization sequence from a monadic tower (see \Cref{Ex_Homology_Loc_Tower}). Unfortunately, our monadic convergence theorem does not give an alternative proof for the convergence of this long homology localization sequence, since it applies to presentable rather than \textit{op}-presentable $\infty$-categories. We shall, however, discuss how one might be able to modify our methods to address this dual setting, modulo a certain closure property of presentable $\infty$-categories, which is known to hold for \textit{ordinary} presentable categories (\Cref{Shauli_Thesis}).   

\subsubsection{Relation to other work}
It was only in an advanced stage of this project, that I have learned that most of the ideas and results presented above were already known for \textit{ordinary} categories (see, for example, \cite{applegate1970iterated,dubuc2006kan,macdonald1982tower,adamek1989monadic}, for various treatments of the subject). However, even in hindsight, the generalization to $\infty$-categories is not so straightforward. While for generalizing the \textit{construction} of the monadic tower, one has to deal ``only'' with the usual higher coherence issues, for proving its \textit{convergence}, one has to deal with a more serious obstacle. Monadic functors of ordinary categories are \textit{faithful}\footnote{This is in accordance with the `stuff, structure, properties' philosophy of \cite[Section 2.4]{baez2010lectures}.}, which allows one to argue about \textit{subobjects} and \textit{quotients}. This phenomenon has no evident analogue in the world of  $\infty$-categories, as a general monadic functor might induce on mapping spaces, maps which are not homotopically truncated at any degree. The argument for convergence presented in this paper is therefore, necessarily, of a different flavor. This is also where the presentability assumption, which is stronger than what is known to be needed for ordinary categories, is most heavily used. 


\subsubsection{Acknowledgments}

The diagrams in this paper were created with the aid of
\href{https://q.uiver.app/}{quiver}. I would like to thank Shay Ben Moshe, Shachar Carmeli, Tomer Schlank, and the entire Seminark group for useful discussions regarding the subject of this work. I would like to thank Shay Ben Moshe and Tomer Schlank also for their comments on an earlier draft. I thank the anonymous referee for valuable suggestions and corrections.  
Finally, I want to express my gratitude to Emmanuel Farjoun, for introducing me to homotopy theory and for his constant support and mentoring throughout the years. 

\subsubsection{Conventions}
We shall generally follow \cite{htt,HA} in notation and terminology regarding $\infty$-categories and, in particular, in the use of the words `small' and `large' in set-theoretical considerations. We use the symbol $\sseq$ to indicate a \textit{full} subcategory and $\subset$ to indicate a \textit{not necessarily full}  subcategory. 
We have the following $\infty$-categories (of $\infty$-categories):
\begin{itemize}
    \item $\widehat{\Cat}$ is the $\infty$-category of large $\infty$-categories (which is itself \textit{very large}).
    
    \item $\Catr \subset \widehat{\Cat}$ is the wide subcategory of large $\infty$-categories and \textit{right} adjoints.
    
    \item $\Prr \sseq \Catr$ is the full subcategory of presentable $\infty$-categories and \textit{right} adjoints.
    
    \item $\Prl \simeq (\Prr)^\op$ is the $\infty$-category of presentable $\infty$-categories and \textit{left} adjoints. 
    
    \item $\Ord \in \widehat{\Cat}_\infty$ is the (large) poset of \textit{small} ordinals.  
    
    \item $\Cat_\infty \in \widehat{\Cat}_\infty$ is the $\infty$-category of \textit{small} $\infty$-categories.
    
    \item $\Spc \sseq \Cat_\infty$ is the $\infty$-category of \textit{small} spaces (i.e. $\infty$-groupoids). 
    
    \item $\Delta \sseq \Cat_\infty$ is the \textit{simplex category}, whose objects we  denote by $[n]:=(0\to1\to\dots \to n)$.
\end{itemize}
As a default, by an `$\infty$-category' we shall mean a \textit{large} $\infty$-category and by `colimits' we shall mean \textit{small} colimits. We shall also identify a collection of objects $S$ in an $\infty$-category $\cC$ with the full subcategory of $\cC$ spanned by $S$.  

\section{Conservativity and Generation}

In this section, we shall study a certain ``epi-mono'' factorization system on $\Prl$ (\Cref{PrL_Epi_Mono}) and show that it induces on $\Prr$, via the equivalence $\Prr \simeq (\Prl)^{\op}$, a ``coreflection-conservative'' factorization system (\Cref{PrR_Factor_Sys}). We shall also discuss a dual construction giving a second pair of factorization systems on $\Prl$ and $\Prr$ (\Cref{PrL_Factor_Sys}) and discuss the technical differences between the two versions. 

\subsubsection{Colimit generators}

We begin with a few observations regarding ``colimit generators'' in presentable $\infty$-categories. 

\begin{defn}
    For an $\infty$-category $\cC$ and a set of objects $S\sseq \cC$, by the subcategory \tdef{generated by $S$ under colimits} (of some given shapes), we mean the smallest full subcategory of $\cC$, which contains $S$ and is closed under all colimits (of the given shapes) which exist in $\cC$. 
\end{defn}

This should be contrasted with the following:
\begin{defn}\label{Def_Cl}
    For an $\infty$-category $\cC$ and a set of objects $S\sseq \cC$, we let $\mdef{\cl{S}}\sseq \cC$ denote the full subcategory spanned by objects, that can be written as a colimit of a diagram with values in $S$.     
\end{defn}

We warn the reader that $\cl{S}$ might \textit{not} be closed under colimits in $\cC$. Thus, to construct the subcategory generated by $S$ under colimits, one might need to iterate the construction $S \mapsto \cl{S}$ (possibly transfinitely) many times.

\begin{example}
    For $\cC = \Ab$, the set $S = \{\ZZ\}$ clearly generates $\Ab$ under colimits. However, the objects of $\cl{S}$ are precisely those abelian groups which have a presentation, in which every relation contains only two generators. These are known as the \textit{simply presented} abelian groups and not all abelian groups are such.
\end{example}

By definition, every presentable $\infty$-category is generated under (small) colimits by a small set. Conversely, 
\begin{prop}\label{Gen_Presentable}
    Let $\cC$ be a presentable $\infty$-category and $S\sseq \cC$ a small set of objects. The subcategory of $\cC$ generated by $S$ under colimits is also presentable. 
\end{prop}
\begin{proof}
    Let $\kappa$ be a large enough regular cardinal, such that $\cC$ is $\kappa$-compactly generated and all objects of $S$ are $\kappa$-compact. Denote by $\cC^\kappa \sseq \cC$ the full subcategory spanned by the $\kappa$-compact objects and by $\cC_0^\kappa\sseq \cC$ the full subcategory generated by $S$ under $\kappa$-small colimits. Since $\kappa$-compact objects are closed under $\kappa$-small colimits, we have $\cC_0^\kappa \sseq \cC^\kappa$, and moreover, this inclusion preserves $\kappa$-small colimits. Hence, the induced functor
    \[
        \Ind_\kappa(\cC_0^\kappa) \to \Ind_\kappa(\cC^\kappa)
    \]
    is fully faithful by \cite[Proposition 5.3.5.11]{htt} and preserves all colimits by \cite[Proposition 5.3.5.13]{htt}. Finally, the $\kappa$-compactly generated (and so in particular presentable) $\infty$-category $\Ind_\kappa(\cC_0^\kappa)$ is precisely the full subcategory of $\cC$ generated by $S$ under colimits.
\end{proof}

We deduce a useful criterion for a set of objects to generate a presentable $\infty$-category under colimits. 
\begin{cor}\label{Gen_Crit}
    Let $\cC$ be a presentable $\infty$-category and $S\sseq\cC$ a small set of objects. The set $S$  generates $\cC$ under colimits, if and only if the collection of corepresentable functors $\{\Map(Z,-)\}_{Z\in S}$ is jointly conservative.    
\end{cor}
\begin{proof}
    Let $\cC_0 \sseq \cC$ be the full subcategory of $\cC$ generated by $S$ under colimits. By \Cref{Gen_Presentable}, the $\infty$-category $\cC_0$ is presentable. Since the inclusion $\cC_0 \into \cC$ is colimit preserving, by the adjoint functor theorem, it admits a right adjoint $G\colon \cC \to \cC_0$. By the Yoneda lemma, $G$ is conservative if and only if the collection of functors $\{\Map(Z,-)\}_{Z\in\cC_0}$ is jointly conservative. However, since $S$ generates $\cC_0$ under colimits, the latter condition is equivalent to the collection of functors $\{\Map(Z,-)\}_{Z\in S}$ being jointly conservative. Finally, since $G$ admits a fully faithful left adjoint, it is conservative if and only if it is an equivalence, which is if and only if $\cC_0 = \cC$.
\end{proof}

\begin{rem}
    It is easy to see that the `only if' part of \Cref{Gen_Crit} holds without any presentability assumptions. However, the `if' part (which is the real essence of the claim) is more subtle (see \cite[Example 4.3 and Remark 4.4]{borger1990total} for an ``almost presentable'' counterexample). 
\end{rem}

In the situation of \Cref{Gen_Crit}, the condition that the collection of functors $\{\Map(Z,-)\}_{Z\in S}$ is jointly conservative is equivalent to the restricted Yoneda functor $\Yo_S \colon \cC \to \Fun(S^\op,\Spc)$ being conservative. If $\Yo_S$ is moreover \textit{fully faithful}, $S$ is said to \textit{strongly generate} $\cC$. This condition is equivalent to the property that for every object $X\in\cC$, we have 
$\colim_{Z \in S_{/X}} Z \iso X.$
Namely, that every $X\in \cC$ is the colimit of the canonical diagram of objects in $S$ mapping to it (see \cite[Section 4.4]{lurie2009infinity}). 

\begin{example}\label{Ex_CSS}
    Let $\cC = \Cat_\infty$, the $\infty$-category of small $\infty$-categories. The singleton set $\{[1]\}$ generates $\Cat_\infty$ under colimits (say, by \Cref{Gen_Crit}), but does not \textit{strongly} generate it. In contrast, the set $\{[n]\}_{n\in\NN}$ does strongly generate $\Cat_\infty$, as the restricted Yoneda functor $\Cat_\infty \to \Fun(\Delta^\op,\Spc)$ is fully faithful. The essential image consists of \textit{complete Segal spaces} (see, for example, \cite[Corollary 4.3.16]{lurie2009infinity}).  
\end{example}

\subsubsection{Factorization systems}

\Cref{Gen_Presentable} allows us to produce the following ``epi-mono'' factorization system on $\Prl$.
\begin{prop}\label{PrL_Epi_Mono}
    Every functor $F\colon \cC \to \cD$ in $\Prl$ can be uniquely factored as 
    \[
        \cC  \oto{\;\mdef{\cl{F}}\;} \mdef{\cD_\infty} \oto{\mdef{F_\infty}} \cD 
        \qin \Prl,
    \] 
    such that 
    \begin{enumerate}
        \item $\cl{F}(\cC)$ generates $\cD_\infty$ under colimits.
        \item $F_\infty$ is fully faithful.
    \end{enumerate}
Moreover, $F_\infty(\cD_\infty)\sseq \cD$ is the subcategory generated under colimits by $F(\cC)$. 
\end{prop}
\begin{proof}
    For existence, let $\cD_\infty \sseq \cD$ be the full subcategory of $\cD$ generated by $F(\cC)$ under colimits. Since $\cC$ is presentable, it has a small set of objects $S$ that generates it under colimits and hence, $F(S)$ is a small set that generates $\cD_\infty$ under colimits (as $F$ is colimit preserving). Thus, by \Cref{Gen_Presentable}, the $\infty$-category $\cD_\infty$ is presentable. 
    Now, the functor $F$ decomposes (uniquely) as a composition 
    $\cC \oto{\;\cl{F}\;} \cD_\infty \oto{F_\infty} \cD,$
    of functors which clearly satisfy conditions (1) and (2) of the claim.  It remains to show that $\cl{F}$ and $F_\infty$ are left adjoints. The functor $\cl{F}$ preserves colimits by construction and the functor $F_\infty$ preserves colimits since by being fully faithful it reflects colimits (see \cite[Proposition 2.4.7]{riehl2018elements}) and $\cD_\infty$ is closed under colimits in $\cD$. Thus, the claim follows 
    by the adjoint functor theorem.
    
    For uniqueness,  by condition (2), we may first assume without loss of generality that $\cD_\infty$ is a full subcategory of $\cD$. Since the inclusion functor $F_\infty\colon \cD_\infty \into \cD$ is colimit preserving, $\cD_\infty$ must be closed under colimits in $\cD$. Therefore, by condition (1), $\cD_\infty$ is precisely the subcategory generated by $F(\cC)$ under colimits, which identifies uniquely $F_\infty$.
    The uniqueness of $\cl{F}$ follows from the fact that fully faithful functors are monomorphisms of $\infty$-categories and hence $F$ factors through $F_\infty$ in a unique way.
\end{proof}

By the equivalence $\Prr \simeq (\Prl)^\op$, \Cref{PrL_Epi_Mono} induces a factorization system on $\Prr$. The right adjoint of a fully faithful functor is, by definition, a \textit{coreflection}.
The following proposition identifies the other class as consisting precisely of the \textit{conservative} right adjoints. 
\begin{prop}\label{Conserv_Generates}
    Let $G\colon \cD \to \cC$ in $\Prr$ with left adjoint $F$. The functor $G$ is conservative, if and only if $F(\cC)$  generates $\cD$ under colimits. 
\end{prop}
\begin{proof}
    Let $S$ be a small set of objects which generates $\cC$ under colimits. By \Cref{Gen_Crit}, the collection of functors $\{\Map(Z,-)\}_{Z\in S}$ is jointly conservative. Hence, by the cancellation property of conservative functors, $G$ is conservative, if and only if the collection of functors 
    \[
        \Map_\cC(Z,G(-)) \simeq \Map_\cD(F(Z),-)
        \quad\text{for}\quad Z \in S,
    \]
    is jointly conservative.
    Using \Cref{Gen_Crit} again, this is if and only if $F(S)$, or equivalently $F(\cC)$, generates $\cD$ under colimits.  
\end{proof}

We thus get the following ``coreflection-conservative'' factorization system on $\Prr$.
\begin{cor}\label{PrR_Factor_Sys}
    Every functor $G\colon\cD\to\cC$ in $\Prr$ can be uniquely factored as
    \[
        \cD \oto{\mdef{G_\infty}} \mdef{\cD_\infty} \oto{\;\mdef{\cl{G}}\;} \cC
        \qin \Prr,
    \]
    such that
    \begin{enumerate}
        \item $G_\infty$ admits a fully faithful left adjoint.
        \item $\cl{G}$ is conservative.
    \end{enumerate}
    Moreover, we can take $\cD_\infty \sseq \cD$ to be the subcategory generated under colimits by the essential image of the left adjoint of $G$ and $\cl{G}$ to be the restriction of $G$ to $\cD_\infty$.  
\end{cor}
\begin{proof}
     By taking left adjoints, this follows immediately from \Cref{PrL_Epi_Mono} and \Cref{Conserv_Generates}.
\end{proof}

\begin{rem}\label{Abstract_Loc}
    In the situation of \Cref{PrR_Factor_Sys}, let $W$ be the collection of morphisms in $\cD$, that are sent to isomorphisms by $G\colon \cD \to \cC$. The functor $G_\infty\colon\cD \to \cD_\infty$ exhibits $\cD_\infty$ as the abstract localization $\cD[W^{-1}]$, so in particular, the $\infty$-category $\cD_\infty$ depends only on $W$ (\cite[\href{https://kerodon.net/tag/02GA}{02GA}]{kerodon}).
\end{rem}

\begin{example}
     The functor $\Omega^\infty\colon \Sp \to \Spc$ factors as the composition of the connective cover functor $\tau_{\ge0}\colon \Sp \to \Sp^\cn$ with the conservative restricted functor $\Omega^\infty\colon\Sp^\cn \to \Spc$. Correspondingly, the essential image of the left adjoint $\Sigma_+^\infty \colon \Spc \to \Sp$ generates under colimits the subcategory of connective spectra $\Sp^\cn \sseq \Sp$. 
\end{example}

\subsubsection{Dual factorizations}

We also have the ``dual'' of \Cref{PrR_Factor_Sys}, producing a \textit{second} factorization system on $\Prl$.
\begin{prop}\label{PrL_Factor_Sys}
    Every morphism in $\Prl$ can be uniquely factored as a reflection followed by a conservative left adjoint.
\end{prop}
\begin{proof}
    For $F\colon \cC \to \cD$ in $\Prl$, we denote by $\mathrm{Iso}_\cD$ the collection of isomorphisms in $\cD$ and we let $W=F^{-1}(\mathrm{Iso}_\cD)$. The class $\mathrm{Iso}_\cD$ is trivially strongly saturated in the sense of \cite[Definition 5.5.4.5]{htt} and is also of small generation (it is generated by the small collection of identity morphisms of a small generating set of $\cD$). Thus, by \cite[Proposition 5.5.4.16]{htt}, the collection $W$ is also strongly saturated and of small generation. Hence, the full subcategory $\cC_\infty \sseq \cC$ spanned by the $W$-local objects is presentable and reflective (see \cite[Proposition 5.5.4.15]{htt}). Using the universal property of the reflection $F''\colon \cC \to \cC_\infty$ (\cite[Proposition 5.5.4.20]{htt}), we can factor $F$ as a composition
    \[
        \cC \oto{F''} \cC_\infty \oto{F'} \cD
        \qin \Prl.
    \]
    It follows by construction, that $F'$ is conservative and hence, we get the desired factorization. The uniqueness is clear. 
\end{proof}

\begin{example}[Homology Localization]\label{Ex_Homology_Loc}
    Given a spectrum $E\in \Sp$, the functor 
    \[
        E\otimes(-)\colon \Sp \to \Sp \qin \Prl
    \]
    factors as a reflection onto the full subcategory of $E$-local spectra $\Sp_E\sseq \Sp$ followed by the restriction of $E\otimes(-)$ to $\Sp_E$, on which it is conservative. Similarly, the composition
    \[
        \Spc \oto{\Sigma_+^\infty} 
        \Sp \oto{E\otimes(-)} 
        \Sp
        \qin \Prl
    \]
    factors through a reflection onto the full subcategory $\Spc_E \sseq \Spc$ of $E$-local \textit{spaces}, followed by a conservative functor. These are the classical stable and unstable homology localizations constructed by Bousfield in
    \cite{bousfield1979localization,bousfield1975localization}.
\end{example}

The factorization system on $\Prl$ given by \Cref{PrL_Factor_Sys} dualizes  to give a second factorization system on $\Prr$. Namely, given $G\colon \cD \to \cC$ in $\Prr$, it can be factored uniquely as a composition
\[
    \cD \oto{G'} \cC_\infty \oto{G''} \cC,  
\]
where $G'$ admits a conservative left adjoint and $G''$ is fully faithful. One can view this as an ``epi-mono'' factorization system on $\Prr$. It would be nice to have a dual for \Cref{Conserv_Generates}, that characterizes functors with a conservative left adjoint in terms of their essential image (in particular, justifying the name ``epi-mono facorization''). For this, consider first the following dual version of \Cref{Gen_Presentable}:

\begin{conj}\label{Shauli_Thesis}
    Let $\cC$ be a presentable $\infty$-category and $\cC_0\sseq \cC$ a full subcategory. If $\cC_0$ is closed under small limits and $\kappa$-filtered colimits for some regular cardinal $\kappa$, then $\cC_0$ is presentable. 
\end{conj}

\begin{rem}
    For ordinary presentable categories, \Cref{Shauli_Thesis} was proven in \cite{adamek1989reflections} and it seems extremely likely to hold for presentable $\infty$-categories as well (though the proof does not generalize in a straightforward way). In particular, by \cite{rosicky2003left}, under a large cardinal axiom known as the \textit{Vop\v{e}nka principle}, every limit closed subcategory of a presentable $\infty$-category is reflective, which implies \Cref{Shauli_Thesis}.
\end{rem}

Given \Cref{Shauli_Thesis}, one can show that a functor in $\Prr$ admits a conservative left adjoint, if and only if its essential image generates the target under small limits and sufficiently filtered colimits.

\begin{prop}\label{PrR_Epi_Mono}
    Let $G\colon \cD \to \cC$ be a functor in $\Prr$ with left adjoint $F$. If for every regular cardinal $\kappa$, the essential image $G(\cD)$ generates $\cC$ under limits and $\kappa$-filtered colimits, then $F$ is conservative. The converse holds assuming \Cref{Shauli_Thesis}.
\end{prop}
\begin{proof}
    Assume first that for every regular cardinal $\kappa$, the essential image $G(\cD)$ generates $\cC$ under limits and $\kappa$-filtered colimits. Let $f\colon X \to Y$ be a map in $\cC$ such that $F(f)$ is an isomorphism in $\cD$, and let $\cC_\infty\sseq \cC$ be the full subcategory of objects $Z\in\cC$ for which the induced map 
    \[
        f^*\colon \Map_\cC(Y,Z) \longrightarrow \Map_\cC(X,Z)
    \]
    is an isomorphism. Since $F(f)$ is an isomorphism, we have $G(\cD)\sseq \cC_\infty$ by the adjunction $F\dashv G$. The full subcategory $\cC_\infty$ is clearly closed under limits in $\cC$. Moreover, if $\kappa$ is large enough so that both $X$ and $Y$ are $\kappa$-compact, then $\cC_\infty$ is also closed under $\kappa$-filtered colimits in $\cC$. By our assumption, it follows that $\cC_\infty = \cC$ and hence $f$ is an isomorphism by the Yoneda lemma. Therefore $F$ is conservative.
    
    Conversely, assume $F$ is conservative and let $\kappa$ be a sufficiently large regular cardinal so that $G$ preserves $\kappa$ filtered colimits. Let $\cC_\infty \sseq \cC$ be the full subcategory generated by $G(\cD)$ under small limits and $\kappa$-filtered colimits. Denoting by $G''\colon \cC_\infty \into \cC$ the fully faithful embedding, the functor $G$ can be uniquely factored as
    \[
        \cD  \oto{\;G'\;} \cC_\infty \oto{G''} \cC
        \qin \Prr.
    \] 
    As in the proof of \Cref{PrL_Epi_Mono}, both functors $G'$ and $G''$ preserve small limits and $\kappa$-filtered colimits. Now, by \Cref{Shauli_Thesis}, the $\infty$-category $\cC_\infty$ is presentable. Hence, we can invoke the adjoint functor theorem to deduce that $G'$ and $G''$ admit left adjoints $F'$ and $F''$ respectively. It follows that $F$ factors as
    \[
        \cC \oto{F''} \cC_\infty \oto{\;F'\;} \cD
        \qin \Prl.
    \]
    Since $F$ is conservative, $F''$ is conservative. However, $F''$ is a reflection, being the left adjoint of the fully faithful functor $G''$, and so is conservative if and only if it is an equivalence. We deduce that $\cC_\infty = \cC$ and thus that $G(\cD)$ generates $\cC$ under small limits and $\kappa$-filtered colimits.
\end{proof}

\section{Monads and Monadicity}

In this section, we review some background material on monads, monadicity and (co)monadic resolutions in $\infty$-categories. Though the material is fairly standard, proofs of some statements, which were hard to locate in the literature, are provided (most notably   \Cref{Super_Sulyma}). The main takeaway from this section is  \Cref{Free2Cmpl} and \Cref{Cmpl2Free}, which provide control from ``below and above'' on the convergence of the monadic tower. We conclude by considering the examples of stable and unstable \textit{homology completion} (\Cref{Nil_Completion,Unstable_Nil_Completion}). 

\subsubsection{Monads and algebras}

We recall from \cite[Section 4.7.3]{HA} the fundamental facts regarding monads and their algebras.
For every $\infty$-category $\cC$, composition of endofunctors induces a canonical monoidal structure on the $\infty$-category
\[
    \End(\cC):=\Fun(\cC,\cC).
\]
A \textit{monad} on $\cC$ is an algebra object in $\End(\cC)$ with respect to the said monoidal structure. 
Using the natural action of $\End(\cC)$ on $\cC$, 
for every such monad $T\in\alg(\End(\cC))$, one can form the $\infty$-category $\Mod_T(\cC)$ of \textit{$T$-algebras}\footnote{The terminology is somewhat unfortunate. It would be more consistent to call these objects \textit{$T$-modules}.} in $\cC$. This construction comes with a forgetful functor $U\colon \Mod_T(\cC) \to \cC$, which is conservative and
admits a left adjoint. The left adjoint takes an object $X\in\cC$ to the object $TX\in\cC$ equipped with the \textit{free $T$-algebra} structure induced from the monad structure of $T$.
Moreover, as explained in \cite[Remark 4.7.3.8]{HA}, the construction of the $\infty$-category of algebras over a monad is \textit{functorial}. Namely, we get a functor 
\[
    \Mod_{(-)}(\cC)\colon
    \alg(\End(\cC))^\op \to \Catr_{/\cC},
\]
taking a monad $T$ on $\cC$ to the forgetful functor $\Mod_T(\cC) \to \cC$.

When $\cC$ is \textit{presentable}, it makes sense to restrict attention to \textit{accessible} monads. That is, to those which preserve $\kappa$-filtered colimits for some sufficiently large regular cardinal $\kappa$. It is a pleasant fact, that in this case the $\infty$-category of $T$-algebras is also presentable.

\begin{prop}[{\cite[Proposition B-6]{gepner2016universality}}]\label{ModT_Presentable}
    Let $T$ be an accessible monad on a presentable $\infty$-category $\cC$. The $\infty$-category $\Mod_T(\cC)$ is also presentable. 
\end{prop}

\begin{rem}
    The proof of \cite[Proposition B-6]{gepner2016universality} shows the following somewhat sharper statement: If $\cC$ is $\kappa$-compactly generated and $T$ is $\kappa$-accessible, then $\Mod_T(\cC)$ is $\kappa$-compactly generated. 
\end{rem}

\subsubsection{Monadic functors}

Every adjunction of $\infty$-categories gives rise to a monad as follows. Let $G\colon \cD \to \cC$ be a right adjoint functor with a left adjoint $F\colon \cC \to \cD$. Using the adjunction data, one can endow the composition of $G$ and $F$ with a monad structure,
\[
    \mdef{T} := GF \qin \alg(\End(\cC)).
\]

Moreover, by \cite[Proposition 4.7.3.3]{HA}, we have a natural factorization
\[
    \xymatrix@C=10ex{ & \Mod_T(\cC)\ar[d]^U\\
    \cD\ar[r]_{G}\ar@{-->}[ru]^{\mdef{\mnd{G}}} & \cC.
    }
\]
We note that since $U$ is conservative, $G$ is conservative if and only if $\mnd{G}$ is conservative. In general, $\mnd{G}$ need not be an equivalence (even if $G$ is conservative). 

\begin{defn}
    If $\mnd{G}$ is an equivalence, $G$ is called \tdef{monadic}.
\end{defn}

The following is a rather degenerate, yet still useful, instance of monadicity:
\begin{prop}\label{Idemp_Monad}
    Let $G\colon \cD \to \cC$ be a right adjoint. If $G$ is fully faithful, then it is monadic. 
\end{prop}

This can be interpreted as saying that a \textit{reflective property} is a special case of an \textit{algebraic structure}. 
\begin{proof}
    Let $F$ be the left adjoint of $G$. Since $G$ is fully faithful, the counit map is an isomorphism  $FG\iso \Id$. Hence, the associated monad $T = GF$ is \textit{idempotent}, in the sense that the multiplication map is an isomorphism  $T^2 \iso T$. Therefore, by \cite[Proposition 4.8.2.4]{HA}, the structure of a $T$-algebra on an object of $\cC$ is just the \textit{property} of being in the essential image of $T$, which is the same as the essential image of $G$. 
\end{proof}

It is useful to have a criterion for recognizing the construction $G\mapsto\mnd{G}$ in an abstract situation. The next proposition says that $\mnd{G}$ is \textit{characterised} by factoring $G$ through a monadic functor with the same associated monad as $G$. 

\begin{prop}\label{Char_Mnd}
    Let $G\colon \cD \to \cC$ be a composition of right adjoint functors
    \(
        \cD \oto{G_1} \cE \oto{G_0} \cC
    \)
    with left adjoints $F_1$ and $F_0$ respectively. If $G_0$ is monadic and the induced map of monads
    \[
        T_0 := G_0F_0 \to G_0(G_1F_1)F_0 =: T
    \]
    is an isomorphism, then there is an equivalence $\cE \simeq \Mod_T(\cC)$ under which $G_0$ corresponds to $U\colon \Mod_T(\cC) \to \cC$ and $G_1$ corresponds to $\mnd{G}\colon \cD \to \Mod_T(\cC)$. 
\end{prop}
\begin{proof}
    By naturality, we have a commutative diagram
    \[
        \begin{tikzcd}
        	\cD &&&& \cE \\
        	& {\Mod_{T}(\cC)} && {\Mod_{T_0}(\cC)} \\
        	&& {\cC.}
        	\arrow["{\mnd{G}}", from=1-1, to=2-2]
        	\arrow["{G_1}", from=1-1, to=1-5]
        	\arrow["{\qquad\mnd{G}_0\!\!\!\!\!}"'{pos=0.65}, from=1-5, to=2-4]
        	\arrow["{\sim}",dashed, from=2-2, to=2-4]
        	\arrow[from=2-2, to=3-3]
        	\arrow[from=2-4, to=3-3]
        	\arrow["G"', shift right=1, curve={height=21pt}, from=1-1, to=3-3]
        	\arrow["{G_0}", shift left=1, curve={height=-21pt}, from=1-5, to=3-3]
            \end{tikzcd}
        \]
        The dashed arrow is an equivalence, since it is induced by an isomorphism of monads $T_0 \iso T$ and the functor $\mnd{G}_0$ is an equivalence, since $G_0$ is assumed to be monadic. Thus, the claim follows.
\end{proof}

The celebrated Barr-Beck-Lurie \textit{monadicity theorem} provides necessary and sufficient conditions for a right adjoint to be monadic. To state it, let us recall the following terminology from  \cite[Definition 4.7.2.2]{HA}. A \textit{split} simplicial object is an augmented simplicial object which admits an extra degeneracy. An important property of a split simplicial object is that it always exhibits the augmentation as the colimit of the simplicial diagram and, moreover, this colimit is preserved by \textit{any} functor. Given a functor $G\colon\cD \to \cC$, a \textit{$G$-split} simplicial object is a simplicial object in $\cD$, such that after applying $G$, it can be extended to a split simplicial object in $\cC$.

\begin{thm}[Monadicity Theorem, {\cite[Theorem 4.7.5.3]{HA}}]\label{Monadicity}
    Let $G\colon \cD \to \cC$ be a right adjoint. The functor $G$ is monadic, if and only if $G$ is conservative and $\cD$ admits, and $G$ preserves,  $G$-split simplicial colimits. 
\end{thm}

\begin{rem}\label{Functorial_Monadicity}
    The monadicity theorem can be viewed as describing the \textit{essential image} of the functor
    \[
        \Mod_{(-)}(\cC)\colon
        \alg(\End(\cC))^\op \to \Catr_{/\cC}.
    \]
    However, this functor is actually \textit{fully faithful} and the operation taking a right adjoint $\cD\to \cC$ to the associated monad $T$ on $\cC$ is its left adjoint reflection (see \cite{haugseng2020lax} and \cite{heine2017equivalence}).
\end{rem}

The monadicity theorem implies the following cancellation property of monadic functors.

\begin{cor}\label{Monadic_Cancel}
    Let $\cD \oto{G} \cC \oto{K} \cC'$ such that $G$ is a right adjoint and $K$ is conservative. If the composition $KG$ is monadic, then $G$ is monadic.
\end{cor}
\begin{proof}
    First, it is clear that if $KG$ is conservative, then $G$ is conservative. Now, every $G$-split simplicial object in $\cD$ is also $KG$-split and hence its colimit exists in $\cD$ and is preserved by $KG$. As any functor, $K$ preserves split simplicial colimits, and since it is conservative, it also reflects them. Thus, $G$ preserves $G$-split simplicial colimits. The result now follows from the monadicity theorem (\Cref{Monadicity}).
\end{proof}

The monadicity theorem is also a useful tool for establishing monadicity in many naturally occurring situations.

\begin{example}\label{Ex_Sp_cn}
    The functor $\Omega^\infty\colon \Sp^\cn \to \Spc$ is conservative and preserves sifted colimits. Hence, by the monadicity theorem, it is monadic.
    Consequently, we have an equivalence of $\infty$-categories 
    \[
        \Sp^\cn \simeq \Mod_{\Omega^\infty\Sigma_+^\infty}(\Spc).
    \]
    The monad $\Omega^\infty\Sigma_+^\infty \in \alg(\End(\Spc))$ is quite complicated however. 
\end{example}

\begin{rem}
    Unlike a general monadic functor, 
    the functor from \Cref{Ex_Sp_cn} preserves \textit{all} sifted colimits. In \cite[Theorem B-7]{gepner2016universality}, it is shown that a monadic functor $\cD \to \Spc$ preserves sifted colimits, if and only if it exhibit $\cD$ as the $\infty$-category of models for a (finitary) \textit{algebraic theory} in the sense of Lawvere. Roughly speaking, this means that objects of $\cD$ are spaces endowed with an algebraic structure, that can be encoded by a collection of \textit{finitary} operations satisfying certain identities (see \cite[Appendix B]{gepner2016universality} for further discussion and examples). In a different direction, one can also characterize the monads on $\Spc$ which come from \textit{$\infty$-operads} as the \textit{analytic} ones (see \cite{gepner2017infty}).
\end{rem}

\subsubsection{coCompleteness}

One can clarify the role played by each of the assumptions in the monadicity theorem by dividing its proof into three steps:
\begin{enumerate}
    \item Assuming only that $\cD$ admits $G$-split simplicial colimits, one shows that $\mnd{G}\colon \cD \to \Mod_T(\cC)$ admits a left adjoint $\mdef{\mnd{F}}\colon \Mod_T(\cC)\to \cD$.
    
    \item Assuming, in addition, that $G$ preserves $G$-split simplicial colimits, one shows that the unit natural transformation $\Id \to \mnd{G}\mnd{F}$ is an isomorphism and hence, that $\mnd{F}$ is fully faithful.
    
    \item Assuming further that $G$ is also conservative, one deduces that $\mnd{G}$ is conservative as well, and hence an equivalence, as it admits a fully faithful left adjoint $\mnd{F}$.
\end{enumerate}

In particular, one can proceed differently after step (1), by asking instead whether the \textit{counit} natural transformation $\mnd{F}\mnd{G}\to \Id$ is an isomorphism. To address this question, consider the associated \textit{comonad}
\[
    \mdef{\comnd}:= FG \qin \coalg(\End(\cD)).
\] 

\begin{defn}\label{Def_Comnd_Cocmpl}
    For a comonad $M$ on an $\infty$-category $\cD$, 
    the \tdef{comonadic (or $M$-)resolution} of an object $Y\in\cD$ is the augmented simplicial object induced by the comonad structure of $\comnd$,
    \[
        \begin{tikzcd}
    	{{}} & {\comnd^3Y} & {\comnd^2Y} & \comnd Y & Y.
    	\arrow[dotted, no head, from=1-1, to=1-2]
    	\arrow[from=1-4, to=1-5]
    	\arrow[shift left=1, from=1-3, to=1-4]
    	\arrow[shift right=1, from=1-3, to=1-4]
    	\arrow[shorten <=4pt, shorten >=4pt, from=1-4, to=1-3]
    	\arrow[shift left=2, from=1-2, to=1-3]
    	\arrow[from=1-2, to=1-3]
    	\arrow[shift left=1, shorten <=4pt, shorten >=4pt, from=1-3, to=1-2]
    	\arrow[shift right=1, shorten <=4pt, shorten >=4pt, from=1-3, to=1-2]
    	\arrow[shift right=2, from=1-2, to=1-3]
    \end{tikzcd}
    \]
    The object $Y$ is called \tdef{($\comnd$-)cocomplete}, if its augmented $\comnd$-resolution is a colimit cone. Namely, if it is the colimit of its $\comnd$-resolution in a canonical way.    
\end{defn}

We observe that after applying $G$, the unit map $\Id \to GF$ provides an extra degeneracy
\[
\begin{tikzcd}
	{{}} & {G\comnd^3Y} & {G\comnd^2Y} & G\comnd Y  & GY.
	\arrow[dotted, no head, from=1-1, to=1-2]
	\arrow[from=1-4, to=1-5]
	\arrow[shift left=1, from=1-3, to=1-4]
	\arrow[shift right=1, from=1-3, to=1-4]
	\arrow[shorten <=4pt, shorten >=4pt, from=1-4, to=1-3]
	\arrow[shift left=2, from=1-2, to=1-3]
	\arrow[from=1-2, to=1-3]
	\arrow[shift left=1, shorten <=4pt, shorten >=4pt, from=1-3, to=1-2]
	\arrow[shift right=1, shorten <=4pt, shorten >=4pt, from=1-3, to=1-2]
	\arrow[shift right=2, from=1-2, to=1-3]
	\arrow[curve={height=18pt}, dashed, from=1-5, to=1-4]
	\arrow[curve={height=18pt}, dashed, from=1-4, to=1-3]
	\arrow[curve={height=18pt}, dashed, from=1-3, to=1-2]
\end{tikzcd}
\]
Hence, the $M$-resolution is always a $G$-split simplicial diagram. In particular, when $G$ is monadic, every object of $\cD$ is $M$-cocomplete by the monadicity theorem. 

\begin{example}
    Let $U\colon\mathrm{Grp} \to \Set$ be the forgetful functor from the (ordinary) category of groups to that of sets. The comonadic resolution of a group $H \in \mathrm{Grp}$ is then the canonical resolution of $H$ by free groups starting with the group freely generated by the underlying set $UH$. Thus, the fact that every group $H$ is the colimit of its canonical free resolution (i.e., it is \textit{cocomplete}) follows from the classical fact that $U$ is \textit{monadic}. 
\end{example}

\begin{rem}\label{Monadic_Colimit}
    A monadic functor $G\colon\cD\to\cC$ is always a conservative right adjoint. For \textit{presentable} $\infty$-categories, this implies that the essential image of the left adjoint $F$ of $G$ generates $\cD$ under colimits (\Cref{Conserv_Generates}). However, the monadicity of $G$ implies that every object of $\cD$ is the colimit of its comonadic resolution. In particular, we get the stronger claim that $\cl{F(\cC)} = \cD$ (in the sense of \Cref{Def_Cl}). This is a rather special property of monadic functors among all conservative right adjoints.
\end{rem}

In general, the colimit of the $M$-resolution can be described in terms of the adjunction $\mnd{F}\dashv \mnd{G}$.
\begin{prop}\label{Super_Sulyma}
    Let $G\colon \cD \to \cC$ be a functor with a left adjoint $F$ and comonad $\comnd=FG$, and assume that $\cD$ admits $G$-split simplicial colimits. For every $Y\in\cD$, the colimit of the $M$-resolution of $Y$, with the induced map to $Y$, is isomorphic to the counit map $\mnd{F}\mnd{G} Y \to Y$. 
\end{prop}

\begin{proof}
    Let $T = GF$ be the monad on $\cC$ and let 
    \[
        U\colon\Mod_T(\cC) \adj \cC\colon \mdef{F_T}
    \]
    be the associated free-forgetful adjunction. Consider the associated comonad $\mdef{\comnd_T} := F_TU$ on $\Mod_T(\cC)$.
    We observe that
    \[
        \comnd = FG \simeq 
        (\mnd{F}F_T)(U\mnd{G}) = 
        \mnd{F}\comnd_T\mnd{G}.
    \]
    Moreover, the unit map $\Id \to \mnd{G}\mnd{F}$ induces for all $n\ge1$, maps of the form
    \[
        \mnd{F}\comnd_T^n\mnd{G} \to 
        (\mnd{F}\comnd_T\mnd{G})^n \simeq
        \comnd^n,
    \]
    which assemble into a map of simplicial objects. In fact, these maps are \textit{isomorphisms}. Indeed, the monad induced by $U\colon \Mod_T(\cC)\to \cC$ is canonically isomorphic to $T=GF$. Namely, when we whisker the unit map $\Id \to \mnd{G}\mnd{F}$ by $F_T$ from the right and $U$ from the left we get an isomorphism of monads (see \Cref{Functorial_Monadicity})
    \[
        UF_T \iso U \mnd{G}\mnd{F}F_T = GF.
    \]
    It follows that the $\comnd$-resolution of an object $Y\in\cD$ is isomorphic to $\mnd{F}$ of the $\comnd_T$-resolution of $\mnd{G}Y$. Since the forgetful functor $U$ is (tautologically) monadic, $\mnd{G}Y$ is the colimit of its $\comnd_T$-resolution
    via the augmentation map. Now, since $\mnd{F}$ is colimit preserving, we get by the above that $\mnd{F}\mnd{G}Y$ is the colimit of the $\comnd$-resolution of $Y$. Finally, the counit map $FG Y \to Y$ factors as a composition of counit maps
    \[
        FG Y \simeq \mnd{F}(F_TU)\mnd{G} Y \to \mnd{F}\mnd{G} Y \to Y,
    \]
    which implies the last part of the claim.
\end{proof}

\begin{cor}\label{Sulyma}
    Let $G\colon \cD \to \cC$ be a functor with a left adjoint $F$ and comonad $\comnd=FG$, and assume that $\cD$ admits $G$-split simplicial colimits. An object $Y\in\cD$ is $\comnd$-cocomplete, if and only if the counit map $\mnd{F}\mnd{G}(Y) \to Y$ is an isomorphism.
\end{cor}

\begin{rem}
    \Cref{Sulyma} implies that if $Y\in\cD$ is $\comnd$-cocomplete, then $\mnd G$ is fully faithful on maps out of $Y$. In
    {\cite[Theorem 3.14]{sulyma2017categorical}}, it is proved that this holds even without any cocompleteness assumptions on $\cD$.
    We note that the results of \cite{sulyma2017categorical} are not phrased in the framework of \cite{HA}, but rather in the \textit{$\infty$-cosmological} framework of \cite{riehl2016Adjunctions}, but luckily, these two frameworks are now known to be equivalent thanks to \cite{haugseng2020lax}.
\end{rem}

As a special case of \Cref{Sulyma}, the counit map $\mnd{F}\mnd{G} \to \Id$ is an isomorphism on all objects in $F(\cC)$.
\begin{prop}\label{Free2Cmpl}
    Let $G\colon \cD \to \cC$ be a functor with a left adjoint $F$, and assume that $\cD$ admits $G$-split simplicial colimits. The counit
    $\mnd{F}\mnd{G} \to \Id$
    is an isomorphism on every object in $F(\cC)$. 
\end{prop}
\begin{proof}
    Let $\comnd=FG$ be the associated comonad on $\cD$. 
    For every $X\in\cC$, the unit map $X \to GF(X)$ induces an extra degeneracy for the augmented simplicial $\comnd$-resolution of $FX$
    \[
    \begin{tikzcd}
    	{{}} & {\comnd^3FX} & {\comnd^2FX} & \comnd FX & FX,
    	\arrow[dotted, no head, from=1-1, to=1-2]
    	\arrow[from=1-4, to=1-5]
    	\arrow[shift left=1, from=1-3, to=1-4]
    	\arrow[shift right=1, from=1-3, to=1-4]
    	\arrow[shorten <=4pt, shorten >=4pt, from=1-4, to=1-3]
    	\arrow[shift left=2, from=1-2, to=1-3]
    	\arrow[from=1-2, to=1-3]
    	\arrow[shift left=1, shorten <=4pt, shorten >=4pt, from=1-3, to=1-2]
    	\arrow[shift right=1, shorten <=4pt, shorten >=4pt, from=1-3, to=1-2]
    	\arrow[shift right=2, from=1-2, to=1-3]
    	\arrow[curve={height=18pt}, dashed, from=1-5, to=1-4]
    	\arrow[curve={height=18pt}, dashed, from=1-4, to=1-3]
    	\arrow[curve={height=18pt}, dashed, from=1-3, to=1-2]
    \end{tikzcd}
    \]
    implying that $FX$ is $\comnd$-cocomplete\footnote{This argument is similar yet different from the one used to show that the $M$-resolution of every object is $G$-split.}. By \Cref{Sulyma}, the counit map $\mnd{F}\mnd{G}(FX) \to FX$ is an isomorphism. 
\end{proof}

Somewhat in the opposite direction, we have the following ``lower bound'' on $F(\cC)$, in terms of objects for which the counit of the \textit{original} adjunction $F G \to \Id$ is an isomorphism. 

\begin{prop}\label{Cmpl2Free}
    Let $G\colon \cD \to \cC$ be a functor with a left adjoint $F$ and assume that $\cC$ and $\cD$ admit all small colimits. Denote by $\cD_0\sseq \cD$ be the full subcategory spanned by objects for which the counit map $F G \to \Id$ is an isomorphism. We have 
    $\cl{\cD}_0 \sseq F(\cC)$. 
\end{prop}
\begin{proof}
    Given an object $Y=\colim Y_k \in \cD$, with $Y_k \in \cD_0$ for all $k$, we have 
    \[
        F(\colim G (Y_k)) \simeq 
        \colim(F G (Y_k)) \simeq
        \colim Y_k \simeq Y.
    \]
    Hence, $Y \in F(\cC)$. 
\end{proof}

\subsubsection{Comonadicity and completeness}

By passing to opposite $\infty$-categories, all the discussion above can be dualized to give analogous constructions and statements for \textit{left} adjoints, \textit{comonads} and their \textit{coalgebras}. The structure of a coalgebra over a comonad can be used to encode ``descent data'' in various (algebro-)geometric situations and the dual of the monadicity theorem can be used to established various descent results (see \cite[Section 4.7]{borceux1994handbook} for a basic introduction and \cite[Appendix D]{SAG} for a comprehensive theory).  
The notion of \textit{completeness}, dual to that of \Cref{Def_Comnd_Cocmpl}, is more familiar in this setting.
\begin{example}[Nilpotent Completion]\label{Nil_Completion}
     Let $R\in\alg(\Sp)$ be a ring spectrum and consider the associated free-forgetful adjunction
    \[
        R\otimes(-)\colon \Sp \adj \Mod_R(\Sp)\colon U.
    \]
    Every spectrum $X\in\Sp$ has a \textit{cosimplicial} $R$-resolution
    \[
        \begin{tikzcd}
        	{R\otimes X} & {R^{\otimes 2}\otimes X} & { R^{\otimes3}\otimes X} & {{}}
        	\arrow[shift left=1, from=1-1, to=1-2]
        	\arrow[shift right=1, from=1-1, to=1-2]
        	\arrow[shorten <=4pt, shorten >=4pt, from=1-2, to=1-1]
        	\arrow[shift left=2, from=1-2, to=1-3]
        	\arrow[from=1-2, to=1-3]
        	\arrow[shift right=2, from=1-2, to=1-3]
        	\arrow[shift left=1, shorten <=4pt, shorten >=4pt, from=1-3, to=1-2]
        	\arrow[shift right=1, shorten <=4pt, shorten >=4pt, from=1-3, to=1-2]
        	\arrow[dotted, no head, from=1-3, to=1-4]
        \end{tikzcd}
    \]
    whose limit is known as the \textit{$R$-nilpotent completion} of $X$ (see \cite[Proposition 2.14]{mathew2017nilpotence} for comparison with the more classical \cite[Definition 1.3]{ravenel1984localization} in terms of the $R$-based Adams tower). An object $X$ is called \textit{$R$-nilpotent complete} if the canonical augmentation of the above co-simplicial diagram exhibits $X$ as its own $R$-nilpotent completion. The dual of \Cref{Free2Cmpl} in this case recovers the standard observation that the underlying spectrum of every $R$-module is $R$-nilpotent complete. We should also note that an $R$-nilpotent complete spectrum is always \textit{$R$-local} in the sense of Bousfield (\Cref{Ex_Homology_Loc}), but the converse need not hold in general. In particular, the (conservative) functor
    \[
        R\otimes(-)\colon \Sp{}_R \to \Mod_R(\Sp)
    \]
    need not be \textit{comonadic}. 
\end{example}

\begin{example}[Unstable Nilpotent Completion]\label{Unstable_Nil_Completion}
    As in \Cref{Nil_Completion}, using the composition of the left adjoint functors
    \[
        \Spc \oto{\:\Sigma_+^\infty\:} 
        \Sp \oto{R\otimes(-)}
        \Mod_R(\Sp),
    \]
    one can similarly define the $R$-nilpotent completion $\widehat{X}_R$ of a space $X$ (classically denoted $R_\infty X$ by Bousfield and Kan in \cite{bousfield1972homotopy}). Spaces for which the $R$-nilpotent completion coincides with the $R$-localization are called \textit{$R$-good} and  otherwise, \textit{$R$-bad} (see, for example,  \cite{bousfield1992completion} for a thorough investigation of $\FF_p$-good and $\FF_p$-bad spaces). 
\end{example}

\section{The Monadic Tower} 

In this section, we construct the monadic tower of a right adjoint functor (\Cref{Def_Monadic_Tower}) and the associated (transfinite) colocalization sequence (\Cref{Def_Coloc_Sequence}). We then prove our main result regarding the convergence of the monadic tower under the assumption of presentability (\Cref{Monadic_Convergence}). For conservative right adjoints, we deduce a characterization as transfinite compositions of monadic functors (\Cref{Trans_Monad}) and for general right adjoints, we deduce an iterated colimit formula for the coreflection part (\Cref{Iterated_Colimit_Functorial}).
We conclude with a discussion about the relation of our results to the long homology localization tower of \cite{dror1977long}.

\subsubsection{Construction}

Under mild cocompleteness assumptions, we can factor a general right adjoint functor $G\colon \cD \to \cC$ through a (possibly transfinite) composition of monadic functors, by iterating the construction $G \mapsto \mnd{G}$. 
Let $\Ord \in \widehat{\Cat}$ be the (large) poset of small ordinals.

\begin{defn}[Monadic Tower]\label{Def_Monadic_Tower} 
    Let $G\colon\cD \to \cC$ in $\Catr$, such that $\cD$ admits small sifted colimits. The \tdef{monadic tower} of $G$ is a diagram,
    \[
        \mdef{\up{\cC}{\bullet}}\colon 
        \Ord^\op \longrightarrow \Catr_{\cD/},
    \]
    which is defined as follows:
    \begin{enumerate}
        \item For the \textit{initial} ordinal $\alpha = 0$, we set $\mdef{\cC_0}:=\cC$ and $\mdef{G_0}:=G$. 
        
        \item For a \textit{successive} ordinal $\alpha + 1$, we first let $\up{F}{\alpha}$ be the left adjoint of $\up{G}{\alpha}$ and define the associated monad
        \[
            \up{T}{\alpha} := 
            \up{G}{\alpha}\up{F}{\alpha} \qin 
            \alg(\End(\up{\cC}{\alpha})).
        \]
        We then extend the tower by
        \[
            \xymatrix@C=10ex{ & \qquad\Mod_{\up{T}{\alpha}}(\up{\cC}{\alpha}) =: \mdef{\cC_{\alpha+1}}\ar[d]^{\up{U}{\alpha}}\\
            \cD\ar[r]_{\up{G}{\alpha}}\ar[ru]^{\mdef{G_{\alpha+1}}:=\mnd{G}_{\alpha}\quad} & \up{\cC}{\alpha}.
            }
        \]
        Since $\cD$ admits simplicial colimits, the functor $G_{\alpha+1}$ is a right adjoint. 
        
        \item For a \textit{limit} ordinal $\alpha$, we define 
        $\mdef{\cC_\alpha} := \invlim_{\beta<\alpha}\up{\cC}{\beta}.$
        Since $\cD$ admits $\alpha$-shaped colimits, by \cite[Theorem B]{horev2017conjugates}, we have an induced right adjoint functor $\mdef{G_\alpha}\colon \cD \to \up{\cC}{\alpha}$ extending the tower constructed thus far. 
    \end{enumerate}
\end{defn}

In the situation of \Cref{Def_Monadic_Tower}, we have for every $\alpha\in\Ord$ also a comonad 
$\mdef{\comnd_\alpha} := \up{F}{\alpha} \up{G}{\alpha}$
on $\cD$ with a counit $\comnd_\alpha \to \Id_\cD$. For successive ordinals, 
the factorization of $\up{G}{\alpha}$ via $\up{G}{\alpha+1}$ induces a map 
$\comnd_\alpha \to \comnd_{\alpha+1}$
compatible with the respective counits. Furthermore, by \Cref{Super_Sulyma}, we have for every $Y\in\cD$ a simplicial colimit diagram (over $Y$)
\[
    \begin{tikzcd}
    	{{}} & {\comnd_{\alpha}^3Y} & {\comnd_{\alpha}^2Y} & \comnd_{\alpha} Y & \mdef{\comnd_{\alpha+1}Y} \\
    	&&&& Y.
    	\arrow[dotted, no head, from=1-1, to=1-2]
    	\arrow[dashed, from=1-4, to=1-5]
    	\arrow[shift left=1, from=1-3, to=1-4]
    	\arrow[shift right=1, from=1-3, to=1-4]
    	\arrow[shorten <=4pt, shorten >=4pt, from=1-4, to=1-3]
    	\arrow[shift left=2, from=1-2, to=1-3]
    	\arrow[from=1-2, to=1-3]
    	\arrow[shift left=1, shorten <=4pt, shorten >=4pt, from=1-3, to=1-2]
    	\arrow[shift right=1, shorten <=4pt, shorten >=4pt, from=1-3, to=1-2]
    	\arrow[shift right=2, from=1-2, to=1-3]
    	\arrow[dashed,from=1-5, to=2-5]
    	\arrow[from=1-4, to=2-5]
    \end{tikzcd}
\]
For a limit ordinal $\alpha$, \cite[Theorem B]{horev2017conjugates} provides a description of the counit map of $\comnd_\alpha$ as the colimit over $\beta<\alpha$ of the counit maps of the $\comnd_\beta$-s. Namely, for every $Y\in\cD$, the following diagram is a colimit diagram in $\cD_{/Y}$:
\[
    \begin{tikzcd}
    	{\comnd_0 Y} & {\comnd_1 Y} & {\dots} & {\comnd_\beta Y} & {\dots } & {\mdef{\comnd_\alpha Y}} \\
    	&&&&& {Y.}
    	\arrow[from=1-1, to=1-2]
    	\arrow[from=1-2, to=1-3]
    	\arrow[from=1-3, to=1-4]
    	\arrow[from=1-4, to=1-5]
    	\arrow[dashed, from=1-5, to=1-6]
    	\arrow[curve={height=12pt}, from=1-1, to=2-6]
    	\arrow[curve={height=6pt}, from=1-2, to=2-6]
    	\arrow[curve={height=6pt}, from=1-4, to=2-6]
    	\arrow[dashed, from=1-6, to=2-6]
    \end{tikzcd}
\]

\begin{defn}[Colocalization Sequence]\label{Def_Coloc_Sequence}
    For every $Y\in\cD$, the \tdef{(long) colocalization sequence} of $Y$ is the $\Ord$-shaped diagram in $\cD_{/Y}$, that extends the above 
    $(\alpha+1)$-shaped diagram over all ordinals $\alpha\in\Ord$.
\end{defn}

\begin{rem}\label{Functorial_Colimit}
    Each term in the colocalization sequence is given by an iterated colimit of objects in the essential image of $F\colon \cC \to \cD$. Moreover, this presentation is \textit{functorial} in the object.  
\end{rem}

\subsubsection{Higher cocompleteness}
To study the convergence of the monadic tower, we first consider the convergence of the colocalization sequence. For this, we introduce the following generalization of the notion of \textit{cocompleteness} from \Cref{Def_Comnd_Cocmpl}.   

\begin{defn}
    In the situation of \Cref{Def_Monadic_Tower},
    we say that $Y\in\cD$ is \tdef{$\alpha$-cocomplete} for $\alpha\in\Ord$, if the $\alpha$-th counit map $\comnd_\alpha Y \to Y$ is an isomorphism and that it is \tdef{$\infty$-cocomplete}, if it is $\alpha$-cocomplete for some $\alpha\in\Ord$.
\end{defn}

\begin{rem}
    By \Cref{Sulyma}, an object in $\cD$ is $1$-cocomplete, if and only if it is $\comnd$-cocomplete in the sense of
    \Cref{Def_Comnd_Cocmpl}.
\end{rem}

For $\infty$-cocomplete objects, the colocalization sequence converges in a very strong sense.  

\begin{prop}\label{Coloc_Seq_Stabilize}
    For an $\alpha$-cocomplete object $Y\in\cD$, the colocalization sequence of $Y$ stabilizes (up to isomorphism) on $Y$ at $\alpha \in \Ord$. 
\end{prop}
\begin{proof}
    If $Y$ is $\alpha$-cocomplete, then in particular $Y\simeq F_\alpha G_\alpha(Y) \in F_\alpha(\cC_\alpha)$. Thus, by \Cref{Free2Cmpl}, the object $Y$ is also $(\alpha+1)$-cocomplete. It follows by 2-out-of-3 that the map $\comnd_\alpha X \to \comnd_{\alpha+1}X$ is an isomorphism. As a transfinite composition of isomorphisms is an isomorphism, we deduce that the colocalization sequence of $Y$ stabilizes on $Y$ from $\alpha$ on.  
\end{proof}

From now on, we shall restrict our attention to \textit{presentable} $\infty$-categories. Every $G\colon \cD \to \cC$ in $\Prr$ is accessible and hence its associated monad on $\cC$ is also accessible. Combining \Cref{ModT_Presentable} with the fact that the inclusion $\Prr \into \Catr$ preserves limits (\cite[Theorem 5.5.3.18]{htt}), we get that the entire monadic tower of $G$ lifts to presentable $\infty$-categories
\[
    \cC_\bullet\colon 
    \Ord^\op \longrightarrow \Prr_{\cD/}.
\]
In this situation, we get that the full subcategory of   $\infty$-cocomplete objects has a familiar description.

\begin{prop}\label{cocomplete_Is_Generated}
    Let $G\colon \cD \to \cC$ in $\Prr$ with a left adjoint $F$. An object of $\cD$ is $\infty$-cocomplete, if and only if it belongs to the subcategory generated under small colimits by $F(\cC)$. In particular, $G$ is conservative, if and only if all objects of $\cD$ are $\infty$-cocomplete.  
\end{prop}
\begin{proof}
    For every $\alpha \in \Ord$, 
    let $\cD_{\alpha} \sseq \cD$ be the full subcategory spanned by the $\alpha$-cocomplete objects.
    On the one hand, by \Cref{Free2Cmpl}, we have $F_\alpha(\cC_\alpha) \sseq \up{\cD}{\alpha+1}.$
    On the other hand, by \Cref{Cmpl2Free}, we have 
    $\cl{\cD}_\alpha \sseq F_\alpha(\cC_\alpha)$ (in the sense of \Cref{Def_Cl}). 
    We therefore get,
    \[
        \cl{\cD}_{\alpha} \sseq 
        \up{F}{\alpha}(\up{\cC}{\alpha}) \sseq
        \cD_{\alpha+1}.
    \]
    Every small diagram in 
    \(
        \cD_\infty =
        \bigcup_{\alpha\in\Ord} \cD_{\alpha}
    \)
    factors through $\cD_\alpha$ for some $\alpha \in \Ord$. Since $\cl{\cD}_{\alpha} \sseq \cD_{\alpha+1}$, we get that $\cD_\infty$ is closed under small colimits. It remains to show that every object of $\cD_\infty$ is generated under colimits by $F(\cC)$.
    Since we also have 
    \(
        \cD_\infty = 
         \bigcup_{\alpha\in\Ord} F_\alpha(\cC_\alpha),
    \)
    it suffices to show, by transfinite induction on $\alpha$, that every object of $F_\alpha(\cC_\alpha)$ is generated under colimits by $F(\cC)$.
    For $\alpha=0$, we have $F_0=F$ and hence the claim holds. Let $\alpha+1$ be a successive ordinal. We recall that every object in  $\cC_{\alpha+1}=\Mod_{T_\alpha}(\cC_\alpha)$
    is a colimit of a diagram of free $T_\alpha$-algebras (see \Cref{Monadic_Colimit}). In addition, the functor $F_{\alpha+1}$ preserves colimits and its value on a free $T_\alpha$-algebra on $X\in\cC_\alpha$ is isomorphic to $F_\alpha(X)$. Thus,
    \[
        F_{\alpha+1}(\cC_{\alpha+1}) \sseq \cl{F_\alpha(\cC_\alpha)}.
    \]
    We deduce, by the inductive hypothesis, that every object of $F_{\alpha+1}(\cC_{\alpha+1})$ is generated under small colimits by $F(\cC)$.
    For a limit ordinal $\alpha$, the explicit description of $F_\alpha$, provided by \cite[Theorem B]{horev2017conjugates}, implies immediately that
    \[
        F_\alpha(\cC_\alpha) \sseq 
        \cl{\bigcup_{\beta<\alpha}F_\beta(\cC_\beta)}.
    \]
    By the inductive hypothesis, every object of $F_\beta(\cC_\beta)$ for $\beta<\alpha$ is generated under colimits by $F(\cC)$, hence so is every object of $F_\alpha(\cC_\alpha)$. Finally, the last claim follows from \Cref{Conserv_Generates}.
\end{proof}

Another consequence of presentability is a uniform bound on the level of cocompleteness of $\infty$-cocomplete objects.

\begin{prop}\label{Uniform_Cmpl}
    Let $G\colon \cD \to \cC$ in $\Prr$. There exists $\mdef{\alpha_0} \in \Ord$, such that every $\infty$-cocomplete object $Y\in \cD$ is already $\alpha_0$-cocomplete. 
\end{prop}
\begin{proof}
    Let $\cD_\alpha \sseq \cD$ be the full subcategory spanned by the $\alpha$-cocomplete objects. By propositions \ref{cocomplete_Is_Generated} and \ref{Gen_Presentable},
    the $\infty$-category $\cD_\infty = \bigcup_{\alpha\in\Ord}\cD_\alpha$ of $\infty$-cocomplete objects is closed under colimits in $\cD$ and is itself presentable. Hence, it is $\kappa$-compactly generated for some small regular cardinal $\kappa$. Thus, there is a small set $S\sseq \cD_\infty$ of $\kappa$-compact objects, such that each object of $\cD_\infty$ is a colimit of a $\kappa$-filtered diagram with values in $S$. In particular, we have $\cl{S} = \cD_\infty$. Let $\alpha_1 \in \Ord$ be large enough such that $S \sseq \cD_{\alpha_1}$ and hence, $\cl{\cD}_{\alpha_1} = \cD_\infty$. We get, by Propositions \ref{Free2Cmpl} and \ref{Cmpl2Free}, that 
    \(
        \cl{\cD}_{\alpha_1} \sseq \cD_{{\alpha_1}+1},
    \)
    so we can take $\alpha_0 = \alpha_1+1$.
    
\end{proof}

\subsubsection{Convergence}

We are now ready to prove our main theorem regarding the convergence of the monadic tower. 
\begin{thm}\label{Monadic_Convergence}
    Let $G\colon\cD \to \cC$ in $\Prr$ with left adjoint $F$.
    \begin{enumerate}
        \item The monadic tower of $G$ stabilizes for $\alpha\gg0$ on $\mdef{\cC_\infty} := \invlim_\alpha \cC_\alpha$.
        \item The induced functor 
            \(
                \cD \to \cC_\infty
            \)
            admits a fully faithful left adjoint $\cC_\infty \into \cD_\infty$.
        \item  The image of $\cC_\infty \into \cD$ is the subcategory generated under small colimits by $F(\cC)\sseq \cD$. 
    \end{enumerate}
\end{thm}
\begin{proof}
    Let $\cD_\infty \sseq \cD$ be the subcategory generated by $F(\cC)\sseq \cD$ under colimits. By \Cref{cocomplete_Is_Generated}, $\cD_\infty$ is also the full subcategory spanned by the  $\infty$-complete objects. Moreover, by \Cref{Uniform_Cmpl}, there exists an ordinal $\alpha \in \Ord$, such that every $\infty$-complete object is already $\alpha$-complete. Consider now the functor $G_\alpha \colon \cD \to \cC_\alpha$. Since composition of conservative functors is conservative, the functor $\cC_\alpha \to \cC$ is conservative. Hence, the functor $G_\alpha$ inverts the same morphisms as the functor $G$. It follows that we have a factorization of $G_\alpha$ as
    \[
        \cD \oto{G_\infty} 
        \cD_\infty \oto{\cl{G}_\alpha} 
        \cC_\alpha,
    \]
    with corresponding left adjoints $F_\infty$ and $\cl{F}_\alpha$. By construction, the $\alpha$-th counit map 
    \(
        F_\alpha G_\alpha(Y) \to Y
    \)
    is an isomorphism for all $Y\in \cD_\infty$ and hence, so is the counit map $\cl{F}_\alpha \cl{G}_\alpha(Y) \to Y$. We deduce that $\cl{G}_\alpha$ is \textit{fully faithful} and so by \Cref{Idemp_Monad}, it is \textit{monadic}. Namely, $\cl{G}_\alpha$ induces an equivalence between $\cD_\infty$ and the $\infty$-category of algebras over the monad $\cl{G}_\alpha \cl{F}_\alpha$ on $\cC_\alpha$. However, since $G_\infty$ is a coreflection, we have an equivalence of monads
    \[
        \cl{G}_\alpha \cl{F}_\alpha \iso
        \cl{G}_\alpha (G_\infty F_\infty) \cl{F}_\alpha \simeq
        G_\alpha F_\alpha.
    \]
    Thus, 
    $G_{\alpha+1}\colon \cD \to \cC_{\alpha+1}$
    factors as (see \Cref{Char_Mnd})
    \[
        \cD \oto{G_\infty} 
        \cD_\infty \iso
        \cC_{\alpha+1}.
    \]
    Since the equivalence $\cD_\infty \iso \cC_{\alpha+1}$ is in particular fully faithful, the tower stabilizes from this point, by the same argument as before, and we get claim (1) and  $\cC_\infty \simeq \cC_{\alpha+1}$. Finally, under the equivalence $\cD_\infty \iso \cC_\infty$, the induced functor $\cD \to \cC_\infty$ is exactly $G_\infty$, which proves (2) and (3).
\end{proof}

\begin{cor}\label{Trans_Monad}
    Let $G\colon\cD \to \cC$ in $\Prr$. If $G$ is conservative, its monadic tower $\up{\cC}{\bullet}$ stabilizes and gives an equivalence 
    \[
        \up{G}{\alpha}\colon \cD \iso \up{\cC}{\alpha}
    \]
    for all $\alpha \gg 0$.
\end{cor}

\begin{proof}
    By \Cref{Monadic_Convergence}, the monadic tower of $G$ stabilizes at some $\alpha\in\Ord$ and the induced functor $G_\alpha \colon \cD \to \cC_\alpha$ is a coreflection. Since $G$ is conservative so is $G_\alpha$ and a conservative coreflection is an equivalence. 
\end{proof}

Since monadic functors are conservative right adjoints and the latter are closed under composition, conservativity is clearly a \textit{necessary} condition for a functor to be a composition of monadic functors. \Cref{Trans_Monad} implies that for presentable $\infty$-categories this condition is also \textit{sufficient}.

The next corollary was suggested to me by Tomer Schlank. Loosely speaking, it says that every presentable $\infty$-category can be viewed as an $\infty$-category of spaces with an \textit{essentially algebraic structure}.
\begin{cor}
    For every presentable $\infty$-category $\cC$, there is an $\alpha \in \Ord$ and a tower 
    \[
        \cC_\bullet \colon 
        \Ord_{\le \alpha}^\op \longrightarrow \Prr,
    \]
    such that
\begin{enumerate}
    \item $\cC_0 = \Spc$ and $\cC_\alpha = \cC$.
    
    \item For every ordinal $\beta < \alpha$, the $\infty$-category $\cC_{\beta+1}$ is monadic over $\cC_\beta$.
    
    \item For every limit ordinal $\beta\le \alpha$, we have 
    $\cC_\beta \simeq \lim_{\gamma < \beta} \cC_\gamma.$ 
\end{enumerate}
\end{cor}

\begin{proof}
    By \Cref{Trans_Monad}, it suffices to construct a conservative right adjoint functor $G\colon\cC \to \Spc$, as then we can take its monadic tower. Since $\cC$ is presentable, it is $\kappa$-compactly generated for some $\kappa$. Thus, there is a small set $S \sseq \cC$ of $\kappa$-compact object, which generates $\cC$ under colimits. It follows, by \Cref{Gen_Crit}, that 
    \[
        \mdef{G} := \prod_{X\in S} \Map(X,-) \colon \cC \longrightarrow \Spc
    \]
    is a conservative right adjoint functor. 
\end{proof}

Consider now a general, not necessarily conservative, functor $G\colon \cC \to \cD$ in $\Prr$. By \Cref{PrR_Factor_Sys}, the functor $G$ admits a coreflection-conservative factorization
\[
    \cD \oto{G_\infty} \cD_\infty \oto{\;\cl{G}\;} \cC.
\]
The left adjoint $F_\infty$ of $G_\infty$ is fully faithful and the idempotent comonad $M_\infty = F_\infty G_\infty$ on $\cD$ provides a coreflection onto $F_\infty (\cD_\infty) \sseq \cD$.
\Cref{Monadic_Convergence} provides the following ``formula'' for this coreflection.

\begin{cor}\label{Iterated_Colimit_Functorial}
    Let $G\colon \cC \to \cD$ in $\Prr$ with a coreflection-conservative factorization
    \[
        \cD \oto{G_\infty} \cD_\infty \oto{\;\cl{G}\;} \cC.
    \]
    There exists $\alpha_0 \in \Ord$, such that for every $Y\in\cD$, the colocalization sequence
    \[
        \comnd_0 Y \lto \comnd_1 Y \lto \dots \lto \comnd_\alpha Y \lto \dots
        \qin \cD
    \]
    stabilizes on $\comnd_\infty Y$ for $\alpha \ge \alpha_0$.
\end{cor}
\begin{proof}
    By \Cref{Monadic_Convergence}, there exists an $\alpha_0\in\Ord$, such that the abstract coreflection $\cD \to \cD_\infty$ is canonically identified with the functor $\cD \to \cC_{\alpha}$, induced by the monadic tower of $G$, for all $\alpha \ge \alpha_0$. Thus, for those $\alpha$-s, the counit maps $\comnd_\alpha Y \to Y$ exhibit $\comnd_\alpha Y$ as the coreflection of $Y$ onto the essential image of $\cD_\infty$. Hence, we get $M_\alpha Y \iso M_\infty Y$ over $Y$, for all $\alpha \ge \alpha_0$.
\end{proof}

Since $F_\infty(\cD_\infty)\sseq \cD$ is the subcategory generated by $F(\cC)$ under colimits, one should think of the coreflection map $\comnd_\infty Y \to Y$ as a universal left approximation of $Y$ by an \textit{iterated} colimit of objects in $F(\cC)$.
Recall that for every $\alpha\in\Ord$, the object $M_\alpha Y$, in position $\alpha$ in the colocalization sequence, is canonically a colimit of a diagram of objects in the essential images of the $F_\beta$-s for $\beta < \alpha$ (see \Cref{Functorial_Colimit}). Thus, \Cref{Iterated_Colimit_Functorial} provides a presentation of $\comnd_\infty Y$ as an iterated colimit of objects in $F(\cC)$, which is canonical and functorial in $Y\in\cD$.

\subsubsection{Homology localization tower}

As always, we can also consider the dual situation. Namely, given a functor $F\colon\cC \to \cD$ in $\Prl$, we can apply the dual of \Cref{Def_Monadic_Tower}, to produce the \tdef{comonadic tower} of $F$, which consists of left adjoints $\mdef{F_\alpha} \colon \cC \to \cD_\alpha$. Similarly, for every $X\in\cC$, the dual of \Cref{Def_Coloc_Sequence} gives the \tdef{long localization tower}
\[
    \begin{tikzcd}
    	& X \\
    	\dots & {T_\alpha X} & \dots & {T_2X} & {T_1X} & {T_0X,}
    	\arrow[from=2-5, to=2-6]
    	\arrow[from=2-2, to=2-3]
    	\arrow[from=2-1, to=2-2]
    	\arrow[from=1-2, to=2-2]
    	\arrow[curve={height=-9pt}, from=1-2, to=2-5]
    	\arrow[curve={height=-12pt}, from=1-2, to=2-6]
    	\arrow[from=2-4, to=2-5]
    	\arrow[from=2-3, to=2-4]
    	\arrow[curve={height=-6pt}, from=1-2, to=2-4]
    \end{tikzcd}
\]
where $\mdef{T_\alpha} := G_\alpha F_\alpha$ are the associated monads on $\cC$.
In view of \Cref{PrR_Epi_Mono}, it seems reasonable to expect that suitable variants of the results of the proceeding section carry over to this dual setting. For specific examples, one can verify by hand that this is indeed the case. We shall recall the one prototypical such example, of Farjoun and Dwyer, which inspired this work.

\begin{example}[Long homology localization tower]\label{Ex_Homology_Loc_Tower}
    For $R\in\alg(\Sp)$, let $R[-]$ denote the composition of the functors 
    \[
        \begin{tikzcd}
        	\Spc & \Sp & \Mod_R(\Sp).
        	\arrow["{\Sigma^\infty_+}", from=1-1, to=1-2]
        	\arrow["{R\otimes(-)}", from=1-2, to=1-3]
        \end{tikzcd}
    \]
    A map of spaces is sent to an isomorphism by $R[-]$, if and only if it is an $R$-homology equivalence. Thus, $R[-]$ factors through $R$-localization $L_R \colon \Spc \to \Spc_R$  (\Cref{Ex_Homology_Loc}). 
    For every space $X\in\Spc$, we get an associated localization tower $X \to T_\bullet X$ of spaces under $X$, which factors through $L_R X$. The first two terms of this tower have familiar descriptions
    \[
        T_0X = \Omega^\infty R[X] \quad,\quad
        T_1 X = \widehat{X}_R,
    \]
    where $\widehat{X}_R$ is the $R$-nilpotent completion of $X$ (\Cref{Unstable_Nil_Completion}). 
    If $X$ is \textit{$R$-good}, then (by definition) the canonical map $L_R X \to \widehat{X}_R$ is an isomorphism and the tower stabilizes. If however $X$ is \textit{$R$-bad}, the tower might not stabilize at $\alpha=1$. 
\end{example}

In \cite{dror1977long}, Dwyer and Farjoun construct the homology localization tower of $X$ with respect to $R$ (for $R\sseq \QQ$ or $R=\FF_p$), without explicit reference to monadic resolutions. Instead, they exploit the fact that the limit of a cosimplicial diagram is isomorphic to the limit of the underling \textit{semi}-cosimplicial diagram. Thus, for a monad $T$, the underling semi-cosimplicial diagram of a $T$-resolution of an object depends only on the unit augmentation $\Id \to T$ and not on the multiplication map $T^2\to T$. In this way, they were able to construct the localization tower directly, using only the augmentations $\Id \to T_\alpha$, using the explicit limit formula for $T_\alpha$ in terms of the $T_\beta$-s with $\beta<\alpha$.\footnote{Note, however, that one can not get the associated monadic tower of \textit{$\infty$-categories} in such a way.} Furthermore, they show in \cite[Proposition 1.2]{dror1977long}, that for every space $X$, there exists $\alpha \in \Ord$, such that the localization tower of $X$ stabilizes at step $\alpha$ on $L_R X$. The proof goes by an explicit analysis of the algebraic structure of the fundamental group and the higher homotopy groups as modules over it.

\bibliographystyle{alpha}
\phantomsection\addcontentsline{toc}{section}{\refname}
\bibliography{TransMonad}

\end{document}